\documentclass{amsart}  
  
\usepackage{citesort}  
\usepackage{enumerate}  
\usepackage{epsfig} 
  
\newtheorem{theorem}{Theorem}[section]  
\newtheorem{lemma}[theorem]{Lemma}  
\newtheorem{corollary}[theorem]{Corollary}

\theoremstyle{definition}  
\newtheorem{definition}[theorem]{Definition}

\theoremstyle{remark}  
\newtheorem{remark}[theorem]{Remark}  
  
\numberwithin{equation}{section}

\def\dt{\frac{d}{dt}}  
\long\def\umbruch{{\displaybreak[1]}}  
\def\R{{\mathbb{R}}}  
\def\N{{\mathbb{N}}}  
\DeclareMathOperator{\graph}{graph}  
\DeclareMathOperator{\tr}{tr}  
\DeclareMathOperator{\divergence}{div}  
\def\epsilon{\varepsilon}  
\def\ol#1{\overline{#1}} 
\newcommand{\Sphere}{\mathbb{S}}  
\renewcommand{\div}{\divergence}  
\newcommand{\K}{\mathcal{K}}
 
\begin{document}  
  
\title{Stability of mean convex cones\\ under mean curvature flow}  
  
\author{Julie Clutterbuck and Oliver C. Schn\"urer}  
\address{Centre for Mathematics and its Applications,  
  Mathematical Sciences Institute,  
  Australian National University, Canberra, ACT 0200, Australia}  
\email{Julie.Clutterbuck@maths.anu.edu.au}  
\address{Freie Universit\"at Berlin, Arnimallee 3, 14195 Berlin, Germany}  
\curraddr{Centre for Mathematics and its Applications,  
  Mathematical Sciences Institute,  
  Australian National University, Canberra, ACT 0200, Australia}  
\def\fuhome{@math.fu-berlin.de} 
\def\anuhome{@maths.anu.edu.au} 
\email{Oliver.Schnuerer\fuhome, 
  Oliver.Schnuerer\anuhome}  
\thanks{Supported by ANU, ARC, AvH foundation and DFG}  
  
  
\date{November 2008.}  
  
\dedicatory{}  
  
  
\begin{abstract}
We consider graphical solutions to mean curvature flow and obtain a
stability result for homothetically expanding solutions coming out of
cones of positive mean curvature: If another solution is initially
close to the cone at infinity, then the difference to the
homothetically expanding solution becomes small for large times.
The proof involves the construction of appropriate barriers. 
\end{abstract}  
  
\maketitle  
  
\markboth{STABILITY OF MEAN CONVEX CONES UNDER MEAN CURVATURE  
FLOW}{JULIE CLUTTERBUCK AND OLIVER C. SCHN\"URER}

\section{Introduction}

We study solutions to graphical mean curvature flow  
\begin{equation}\label{mcf}  
\dot u=\sqrt{1+|Du|^2}\,\,  
\div\left(\frac{Du}{\sqrt{1+|Du|^2}}\right)  
\equiv\sqrt{1+|Du|^2}\,\,H[u]  
\end{equation}  
for functions $u\in C^\infty_{loc}\left(\R^n\times(0,\infty)\right)
\cap C^0_{loc}\left(\R^n\times[0,\infty)\right)$. This equation is
known to have a solution for initial data $u(\cdot,0)\in
C^0_{loc}\left(\R^n\right)$. Let $k:\R^n\to\R$ be smooth outside the
origin and positive homogeneous of degree one. Then $\graph
k\subset\R^{n+1}$ is a cone. The unique solution $U$ to \eqref{mcf}
with $U(\cdot,0)=k$ is homothetically expanding, which implies that
for any $t_1,\,t_2>0$, $\graph U(\cdot,t_1)$ and $\graph U(\cdot,t_2)$
differ only by a homothety. Hence $U$ fulfills
\begin{equation}\label{U scaling}  
U(x,t)=\sqrt{2nt}\,\,U\left(\frac x{\sqrt{2nt}}, \frac1{2n}\right)  
=\sqrt t\,\, U\left(\frac x{\sqrt t},1\right).  
\end{equation}  
We refer to \cite{EckerHuiskenInvent,StavrouSelfSim} for details.  
  
In this paper, we are concerned with solutions $u$ to \eqref{mcf} such  
that $u(\cdot,0)=u_0$ is \emph{close to $k$ at infinity,}  
\begin{equation}\label{close to k eq}  
\sup\limits_{\R^n\setminus B_r(0)}|u_0-k|\to0\quad\text{as}\quad  
r\to\infty.  
\end{equation}  
  
In this situation, we study stability of $U$ under mean curvature  
flow, see our main result, Theorem \ref{main thm}. It implies in  
particular the following  
\begin{theorem}\label{H pos thm}  
Let $k$, $U$, $u_0$ and $u$ be as above. Assume that $\graph k$ is 
contained in a half-space and has positive mean curvature outside the 
origin, $H[k]>0$. Then 
$$\sup\limits_{\R^n}|u(\cdot,t)-U(\cdot,t)|\to0\quad\text{as}\quad  
t\to\infty.$$  
\end{theorem}  
  
The results of this paper hold for the following class of mean convex
cones:
\begin{definition}\label{K def}  
A function $k:\R^n\rightarrow \R$ is said to be of class $\K$ if the
following conditions are fulfilled:
\begin{enumerate} [(i)]
\item \label{k i} $k$ is positive homogeneous of degree one;
\item \label{k ii} $k$ is smooth outside the origin;
\item \label{k iii} $k$ has non-negative mean curvature $H[k]$ outside
  the origin;
\item \label{k iv} there exists a linear function $l$ such that $k\ge
  l$;
\item \label{k v} \label{stability condition} if $n\ge3$, at points
  $p=(\hat{p},p^{n+1})\neq0$ in $\graph k$ where $H[k](p)=0$, we
  require that the second fundamental form $A$ of $\graph k$ fulfills
  \[|A|^2(p)<\left(\tfrac{n-2}2\right)^2|p|^{-2}.\] 
\end{enumerate}
We will refer to both the function $k$ and the hypersurface $\graph k$
as cones.
\end{definition}

Our main theorem is   
\begin{theorem}\label{main thm}
Let $U$ be the homothetically expanding solution to Equation
\eqref{mcf} with $U(\cdot,0)=k$, where $k$ is a cone of class $\K$.
Suppose that $u\in C^\infty_{loc}\left(\R^n\times(0,\infty)\right)
\cap C^0_{loc}\left(\R^n\times[0,\infty)\right)$ solves \eqref{mcf}
with initial data $u_0\in C^0_{loc}\left(\R^n\right)$ approaching $k$
at infinity by fulfilling \eqref{close to k eq}.  Then
$$u(\cdot,t)-U(\cdot,t)\to0 \quad\text{as}\quad t\to\infty,$$  
uniformly in $C^k\left(\R^n\right)$ for every $k\in\N$.     
\end{theorem}

The cones $k$ considered in Theorem \ref{main thm} are such that the
homothetically expanding solutions $U$ studied in
\cite{EckerHuiskenInvent,StavrouSelfSim} fulfill $U(\cdot,t)\ge k$ for
every $t\ge0$. It is clear that this requires $H[k]\ge0$ outside the
origin. We also want to impose a non-negativity condition on $H[k]$ at
the origin. The definition of viscosity solutions suggests a
requirement that $\graph k$ is contained in a half-space; this is
implied by Condition \eqref{k iv} in Definition \ref{K def}. Note that
this condition may be violated even if $H[k]\ge0$ outside the
origin. A counterexample is the higher dimensional analogue of what we
try to illustrate in the picture. The black region is given by
$$\left\{(x,z)\in\R^n\times\R:z>k(x)\right\}\cap \Sphere^n.$$ It is a
starshaped subset of the sphere with respect to the north pole as its
boundary is given by $\graph k$, intersected with $\Sphere^n$. In
higher dimensions, the black region is constructed as follows. Attach
sets of the form $B_r\times[0,1]$ to a geodesic ball around the
origin.  We smooth the resulting set, especially near $B_r\times\{0\}$
and $B_r\times\{1\}$.  As there are hypersurfaces of positive mean
curvature that contain ``necks'' \cite{HuiskenSinestrari3}, we can
ensure that the boundary of the constructed set has positive mean
curvature. If we attach enough sets of the form $B_r\times[0,1]$ that
extend over the equator, it is easy to see that the resulting set is
not contained in a half-space any more. Hence the corresponding
assumption on $k$ is not redundant.
\begin{figure}[htb] 
\epsfig{file=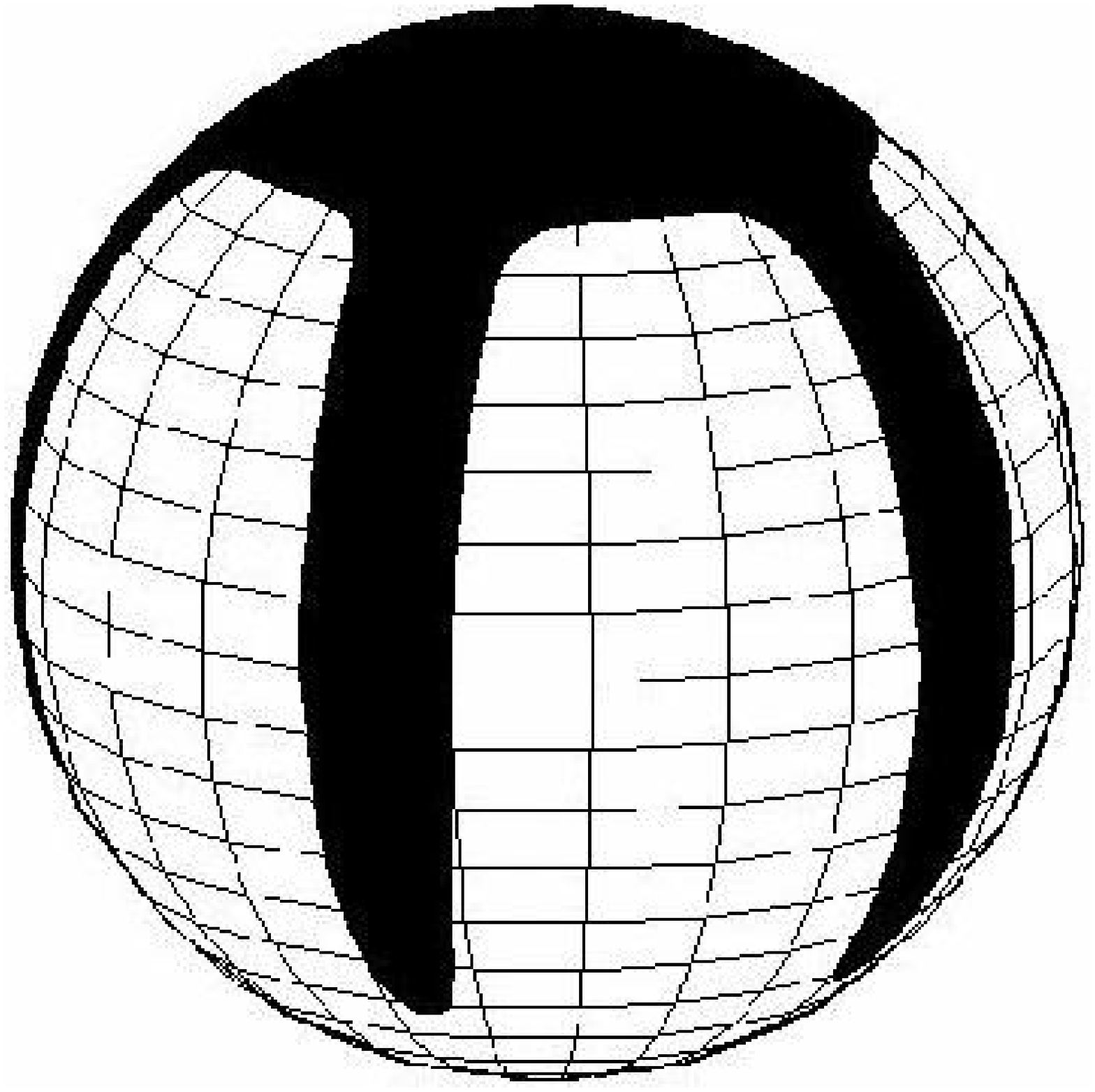,   width=0.4\textwidth} 
\label{starfish wrapped around a ball} 
\end{figure} 
 
Homogeneous minimal cones fulfill a linearized stability condition if 
and only if $|A|^2(p)\le\left(\frac{n-2}2\right)^2|p|^{-2}$, see 
\cite[Example 4.7]{CaffarelliHardtSimonSing} for details. Hence it is 
not surprising that we have to impose such a bound at those points, 
where the mean curvature $H$ vanishes. 

If $k$ is convex, Condition \eqref{stability condition} in Definition
\ref{K def} always holds.
 
Our proof of Theorem \ref{main thm} also works if $k\in C^2_{loc}$ 
away from the origin. We conjecture that the result extends also to 
continuous cones, but expect that this will be quite technical. 
 
The existence of homothetically expanding solutions to mean curvature
flow starting from a cone was first examined by K. Ecker and
G. Huisken in \cite{EckerHuiskenInvent} and further investigated by
N. Stavrou \cite{StavrouSelfSim}. For graphical initial data, the
existence of solutions to \eqref{mcf} is studied by T. Colding and
W. Minicozzi and others in
\cite{EckerHuiskenInvent,JCPhD,ColdingMinicozzi}. After appropriate
rescaling, solutions which deviate initially sublinearly from a cone,
converge for large times to the homothetically expanding
solution. This is proved in
\cite{EckerHuiskenInvent,StavrouSelfSim}. The present paper addresses
the corresponding convergence result without rescaling. Such questions
have also been addressed in
\cite{OSPierreCrelle,OSAlbert,JCOSFSMCFStability,OSFSMSStabilityRicci}.
 
As in those papers addressing stability without rescaling, we have to
impose a decay condition, in our case \eqref{close to k
eq}. Boundedness of the perturbation does at most imply subsequential
convergence to a homothetically expanding solution as appropriate
initial oscillations of $u_0-k$ in space may yield oscillations of
$u(0,t)-U(0,t)$ in time.  Note however, that we do not have to impose
a decay rate in \eqref{close to k eq}. Moreover, using the clearing
out lemma we can weaken \eqref{close to k eq} by allowing small
additional BV-perturbations of $k_0$. Compare this with \cite[Theorem
1.6]{OSFSMSStabilityRicci}.
 
The idea of the proof of Theorem \ref{main thm} and the organization 
of the rest of the paper are as follows. Here we will only sketch the 
proof in the case that $U(\cdot,t)>k$ for all $t>0$. Otherwise, $k$ is 
 linear. We address this special situation in Appendix \ref{Rn 
stable}.\par If $u_0\ge k$, we can use the homothetically expanding 
solutions $U$ as barriers.  For arbitrary $\epsilon>0$ and $T>0$ 
chosen appropriately, we have 
$$U(x,t)-\epsilon\le u(x,t)\le U(x,t+T)+\epsilon$$ for all $(x,t)$. As  
$U(x,t+T)-U(x,t)$ converges to zero as $t\to\infty$ and $\epsilon>0$  
is arbitrary, convergence follows. We will discuss that in detail in  
Section \ref{U barrier sec}.\par If $u_0<k$ somewhere, we construct  
barriers that show that for every $\delta>0$, there exists  
$t_\delta>0$ such that $u(\cdot,t_\delta)\ge k-\delta$. Then the above  
argument can be applied with $\epsilon=2\delta$ and  
$$U(x,t)-\epsilon\le u(x,t+t_\delta)\le U(x,t+T)+\epsilon.$$ The 
barrier construction and details for that part are to be found in 
Section \ref{barr sec}. In the case of convex cones, we can use 
hyperplanes as considered in Appendix \ref{Rn stable} to ensure the 
existence of such a time $t_\delta$ for every $\delta>0$. The rest of 
the argument is similar to the one given above. In Section \ref{BV 
sec} we explain how to weaken the decay condition to allow additional 
decaying perturbations in $BV_{loc}$. \par In the appendices, we 
study stability of hyperplanes and the uniform convergence for a 
family of perturbations. 
 
We want to thank Ben Andrews, Gerhard Huisken and Felix Schulze for
discussions. We gratefully acknowledge support from the Alexander von
Humboldt foundation, the Australian National University, the
Australian Research Council and the Deutsche Forschungsgemeinschaft.
  
\section{One-Sided Perturbations}  
\label{U barrier sec}  
 
Let us first show that the conditions imposed on $k$ ensure that the
homothetic solution always lies above the cone.
\begin{lemma}\label{U above k} 
Assume that $k:\R^n\to\R$ satisfies conditions \eqref{k i}--\eqref{k
iv} of Definition \ref{K def}.  Let $U$ be the homothetically
expanding solution to \eqref{mcf} with $U(\cdot,0)=k$. Then
$U(\cdot,t)\ge k$ for all $t\ge0$.
\end{lemma} 
\begin{proof} 
Mean curvature flow of smooth compact manifolds preserves the
condition $H\ge0$. Here, however, we consider a noncompact manifold
and the existence proof in \cite{EckerHuiskenInvent} involves a
mollification of the initial data. This mollification, however, might
destroy the condition $H\ge0$. Hence the result does not seem to be
trivial.

According to Appendix \ref{Rn stable}, we have $U(0,t)\ge k(0)=0$ for
all $t\ge0$. We may assume that $k$ is not a linear function for
otherwise $U(\cdot,t)=k$ for all $t$. As $U(\cdot,t)$ is smooth for
$t>0$ but $k$ is singular at the origin, we deduce that $U(0,t)>0$ and
$H[U](0,t)>0$ for $t>0$.  Both inequalities extend to a possibly
time-dependent neighborhood of the origin. Near spatial infinity,
comparison with spheres shows that $\sup\limits_{\R^n\setminus
B_r(0)}|U(\cdot,t)-k|\to0$ uniformly as $r\to\infty$ for $t$ in a
bounded time interval. Fix $\epsilon>0$ and assume that there exists
$(x_0,t_0)$ such that $U(x_0,t_0)-k(x_0)\le-\epsilon$. The behavior of
$U$ at spatial infinity ensures that a negative infimum of
$U(\cdot,t)-k$ is attained. Hence we may assume that $(x_0,t_0)$ is
such that $t_0>0$ is minimal with $U(x_0,t_0)-k(x_0)\le-\epsilon$. The
considerations above imply that $|x_0|>0$. As $H[k](x_0)\ge0$, this
contradicts the strong maximum principle applied to $U$ and $k$. The
claim follows.
\end{proof} 
  
\begin{corollary}\label{U strict above k} 
Let $k:\R^n\to\R$ be as in Lemma \ref{U above k}. If $k$ is not a
 linear function, then $U(\cdot,t)>k$ for $t>0$. 
\end{corollary} 
\begin{proof} 
According to the proof of Lemma \ref{U above k}, we have
$U(\cdot,t)\ge0$ and $U(0,t)>0$ for $t>0$. Hence the strong maximum
principle implies that $U(\cdot,t)>0$ for $t>0$. The claim follows.
\end{proof} 

In fact $\dot U>0$, or equivalently $H>0$, for $t>0$. This follows by
adapting techniques of B. White \cite{WhiteSize}. We do not need this
result for proving our Main Theorem \ref{main thm}. Therefore we will
only sketch the proof.
\begin{lemma}
Let $U$ be as in Lemma \ref{U strict above k}. Then $\dot
U(\cdot,t)>0$ for $t>0$.
\end{lemma}
\begin{proof}[Sketch of proof:]
Recall that \cite[Theorem 3.1]{WhiteSize} asserts for compact mean
convex level set solutions $F_t(K)$ to mean curvature flow that
$F_{t+h}(K)\subset\text{interior }F_t(K)$ for every $t,h>0$.

We proceed as in the proof of \cite[Theorem 3.1]{WhiteSize}, with the
following modifications:
\begin{itemize}
\item Remove ``compact'' and ``interior''. 
\item Consider epigraphs, i.\,e. $F_t(K)=\left\{\left(\hat
  x,x^{n+1}\right)\in\R^{n+1}\,:\, x^{n+1}\ge u(\hat x,t)\right\}$.
\item Observe that the semi-group property, \S 2.1 (4), follows as
  solutions to \eqref{mcf} for smooth initial data, which we have for
  positive times, with uniformly bounded gradient are unique.
\item Property \S 2.1 (6) is fulfilled if we construct solutions as in
  \cite{EckerHuiskenInvent}. Observe in particular that $u\le v$ is
  preserved under mollifications.
\end{itemize}
With these modifications, \cite[Theorem 3.1]{WhiteSize} extends to our
situation, i.\,e. $U(\cdot,t_1)\le U(\cdot,t_2)$ for $0\le t_1\le
t_2$.  Hence $\dot U(\cdot,t)\ge0$. As $\dot U(0,t)>0$ for $t>0$, the
strong maximum principle implies that $\dot U(\cdot,t)>0$.
\end{proof}
 
The difference between homothetic solutions starting at different  
times tends to zero for large times.  
\begin{lemma}\label{U converges lem}  
Let $k:\R^n\to\R$ be continuous and positive homogeneous of degree  
one. Let 
$U\in C^\infty_{loc}\left(\R^n\times(0,\infty)\right) \cap  
C^0_{loc}\left(\R^n\times[0,\infty)\right)$ be the homothetically  
expanding solution to \eqref{mcf} with $U(\cdot,0)=k$. Let $T>0$. Then  
$$U(\cdot,t+T)-U(\cdot,t)\to0\quad\text{as}\quad t\to\infty,$$  
uniformly in $C^k$ for any $k\in\N$.  
\end{lemma}  
\begin{proof}  
We will prove that  
$$\Vert U(\cdot,t+T)-U(\cdot,t)\Vert_{L^\infty\left(\R^n\right)}\to0
\quad\text{as}\quad t\to\infty.$$ Then uniform gradient estimates and
local higher derivative estimates, see \cite[Theorems 2.3, 3.1 and
3.4]{EckerHuiskenInvent}, imply the claimed convergence.\par According
to \cite[Theorem 3.1]{EckerHuiskenInvent} and \cite[Corollary
1]{StavrouSelfSim}, we deduce that $|H[U(\cdot,1)]|\le c$. Hence
\eqref{U scaling} implies $|H[U(\cdot,t)]|\le \frac c{\sqrt t}$ and,
as $DU$ is uniformly bounded for $t\ge1$, $\big|\dot U(\cdot,t)\big|
\le\frac c{\sqrt t}$. We integrate from $t$ to $T+t$ and obtain the
claimed convergence.
\end{proof}  
  
As a consequence, we can prove Theorem \ref{main thm} if (the graphs  
of) $u_0$ and $U$ lie on the same side of $k$.  
\begin{proof}[Proof of Theorem \ref{main thm}, 1st part:]  
We will prove Theorem \ref{main thm} under three additional 
assumptions: 
\begin{enumerate}[(i)]  
\item\label{ass i} We have $U(\cdot,t)>k$ for any $t>0$. (This means 
  that $k$ is singular at the origin, see Corollary \ref{U strict 
  above k}. We consider linear functions $k$ in Appendix \ref{Rn 
  stable}.) 
\item\label{ass ii} For every $\delta>0$, there exists $t_\delta>0$  
  such that $u(\cdot,t_\delta)\ge k-\delta$.  
\item $n\ge3$. 
\end{enumerate}  
Let $\delta>0$ and $\epsilon=2\delta$. We obtain  
\begin{equation}\label{U as barrier eq}  
u(\cdot,0+t_\delta)-(U(\cdot,0)-\epsilon)  
=u(\cdot,0+t_\delta)-k+2\delta\ge\delta>0  
\end{equation}  
on $\R^n$. The functions $u$ and $U-\epsilon$ evolve by
\eqref{mcf}. Using small spheres as barriers, the compact maximum
principle implies that $u(\cdot,t+t_\delta)-(U(\cdot,t)-\epsilon)>0$
for some time interval $[0,\zeta]$, $\zeta>0$. By comparison with
large spheres and the compact maximum principle, we see that $u$ and
$U-\epsilon$ grow at most polynomially at spatial infinity. On any
bounded time interval of the form $[\zeta,T]$, the interior estimates
of \cite{EckerHuiskenInvent} imply uniform gradient bounds for
$U(\cdot,t)-\epsilon$. Hence the comparison principle of G. Barles,
S. Biton, M. Bourgoing and O. Ley, see \cite{BarlesetalUniqueness} or
\cite[Theorem A1]{JCOSFSMCFStability}, is applicable and implies that
$u(\cdot,t+t_\delta)-(U(\cdot,t)-\epsilon)\ge0$.\par On the other
hand, we have $u_0\le k+\epsilon$ in $\R^n\setminus B_r(0)$ for $r$
sufficiently large. According to \eqref{U scaling}, we find $T>0$ such
that $U(\cdot,T)\ge u_0$ in $B_r(0)$. Hence \eqref{close to k eq}
implies that $U(\cdot,T)+\epsilon\ge u_0$ in $\R^n$. As above, the
maximum principle implies that $U(x,t+T)+\epsilon\ge u(x,t)$ for all
$(x,t)$, $t\ge0$.\par Combining the above estimates, we get
$$U(\cdot,t)-\epsilon\le u(\cdot,t+t_\delta)\le  
U(\cdot,t+t_\delta+T)+\epsilon.$$  
Lemma \ref{U converges lem} implies that  
$$|u(\cdot,t)-U(\cdot,t)|\le3\epsilon$$ for $t$ sufficiently large. As  
$\epsilon>0$ was arbitrary, we obtain $C^0$-convergence. According to  
\cite{EckerHuiskenInvent}, higher derivatives are uniformly  
bounded. Hence interpolation inequalities imply the claimed  
convergence.  
\end{proof}  
  
\section{Barrier Construction}  
\label{barr sec}  
  
Let us describe the idea of the barrier construction in the case of  
Theorem \ref{H pos thm}. The proof of the corresponding statement for  
Theorem \ref{main thm} is more complicated. We give it below.  
  
Assume that $k$ is as in Theorem \ref{H pos thm}. Consider  
$w:=k-|x|^{-\alpha}$ for $\alpha>0$. We wish to show that $H[w](x)>0$  
for $|x|$ sufficiently large. Consider $H[w]\sqrt{1+|Dw|^2}$.  
According to the scaling behavior of the mean curvature of cones,  
there exists some $\epsilon>0$ such that $H[k](x) \sqrt{1+|Dk|^2}  
\ge\epsilon|x|^{-1}$. A direct calculations using the fundamental  
theorem of calculus, however, shows that $|H[w]\sqrt{1+|Dw|^2}  
-H[k]\sqrt{1+|Dk|^2}|\le c|x|^{-\alpha-2}$. Hence we obtain  
$H[w](x)\ge0$ for $|x|$ sufficiently large as claimed.  
  
In the case of Theorem \ref{main thm}, we obtain the barrier via a  
flow equation.  
\begin{lemma}  \label{barrier} 
Let $k$ be as in Theorem \ref{main thm} and $n\ge3$. Let $X_0$ denote
the embedding vector of $\graph k$ outside the origin and $\nu$ its
downwards pointing unit normal. Deform the embedding vector $X$
according to
$$\dt X=-F\nu,\quad X(\cdot,0)=X_0,\quad F(X)  
=-\left(|X|^2\right)^{-\alpha} \equiv-|X|^{-\frac{n-2}2}.$$ Then for  
$r$ sufficiently large, the image of $X(\cdot,1)$ in $(\R^n\setminus  
B_r(0))\times\R$ can be written as $\graph b$, $b:\R^n\setminus  
B_r(0)$, where $b$ fulfills  
\begin{enumerate}[(a)]  
\item $b\in C^\infty_{loc}$,  
\item $\sup\limits_{\R^n\setminus B_R(0)}|b-k|\to0$ as $R\to\infty$,  
\item $b<k$,  
\item $H[b]>0$.  
\end{enumerate}  
\end{lemma}  
\begin{proof}  
The evolution equation is a first order partial differential equation.
Standard results (for example, iterated application of the results in
\cite[\S{}3.2.4]{EvansPDE}) and the evolution equations below ensure
existence of a solution on $\left(\R^n\setminus
B_R(0)\right)\times[0,1]$ for some large $R$. Hence $b$ is smooth, if
we can write $X(\cdot,1)$ as a graph. Standard methods (see, for
example, \cite{OSA2,HuiskenRoundSphere,CGCPBook}; we also use the
notation used there) yield the following evolution equations
\begin{align*}  
\dt g_{ij}=&\,-2Fh_{ij},\umbruch\\ \dt  
h_{ij}=&\,F_{;ij}-Fh^k_ih_{kj}\umbruch\\  
=&\,4\alpha(\alpha+1)F|X|^{-4}\langle X,X_i\rangle\langle X,X_j\rangle  
-2\alpha F|X|^{-2}(g_{ij}-\langle X,\nu\rangle h_{ij})\\  
&\,-Fh^k_ih_{kj},\umbruch\\ \dt  
H=&\,\dt\left(g^{ij}h_{ij}\right)=-g^{ik}g^{jl}h_{ij}\dt g_{kl}  
+g^{ij}\dt h_{ij}\umbruch\\ =&\,(-F)\left(-|A|^2  
-4\alpha(\alpha+1)|X|^{-4} \left(|X|^2-\langle  
X,\nu\rangle^2\right)\right)\umbruch\\  
&\,+(-F)\left(2\alpha|X|^{-2}(n-\langle X,\nu\rangle  
H)\right),\umbruch\\  
\dt|A|^2=&\,\dt\left(g^{ij}h_{jk}g^{kl}h_{li}\right)  
=-2g^{ir}g^{js}h_{jk}g^{kl}h_{li}\dt g_{rs} +2g^{ij}h_{jk}g^{kl}\dt  
h_{li}\umbruch\\ =&\,2F\tr A^3-8\alpha(\alpha+1)(-F)|X|^{-4}\langle  
X,X_i\rangle h^{ij} \langle X_j,X\rangle\\  
&\,+4\alpha(-F)|X|^{-2}\left(H-\langle  
X,\nu\rangle|A|^2\right),\umbruch\\  
\dt\nu^\beta=&\,F_ig^{ij}X^\beta_j\umbruch\\  
=&\,2\alpha|X|^{-2\alpha-2}\left(X^\beta-\langle  
X,\nu\rangle\nu^\beta\right).  
\end{align*}  
We have the following geometric scaling: $|A|[k](p)\sim|p|^{-2}$ and
$H[k](p)\sim|p|^{-1}$.  Initially, we have $|A|^2(p)\le c_A|p|^{-2}$
for some constant $c_A>0$. As long as $|A|^2(p)\le 2c_A|p|^{-2}$, we
obtain
\begin{align}   \notag
\left|\dt|A|^2\right|\le&\,c|F||A|^3+c|F||X|^{-2}|A|
+c|F||X|^{-2}\left(|A|+|X||A|^2\right)\umbruch\\
\le&\,c|F||A|^3+c|F||X|^{-2}|A|\label{A evol i}\umbruch\\ 
\le&\,c |F| |X|^{-3} \le \frac c{|X|}|X|^{-2}.  \label{A evol}
\end{align}  
For the rest of the proof, we will always assume that $r$ is
sufficiently large, i.\,e.{} our conclusions hold in
$\left(\R^n\setminus B_r(0)\right)\times\R$.  The evolution equation
above justifies our assumption $|A|^2(p)\le 2c_A|p|^{-2}$ for $t\le1$;
hence we will assume $t\in [0,1]$.

The stability condition imposed on $k$ (Condition \eqref{stability
condition} of Definition \ref{K def}) ensures that there exists a
neighbourhood $\mathcal{N}\subset\R^{n+1}$ of the set on which $H=0$
which is invariant under homotheties and translations parallel to
$e_{n+1}$, and an $\epsilon>0$, such that
$|p|^2|A|^2(p)\le\left(\frac{n-2}2\right)^2-\epsilon$ on
$\mathcal{N}$. Now \eqref{A evol} ensures that $|p|^2|A|^2(p) <
\left(\frac{n-2}2\right)^2$ holds on
$\mathcal{N}\cap\left(\left(\R^n\setminus
B_r(0)\right)\times\R\right)$ for $r$ sufficiently large.
In $\mathcal N$, the maximum principle applied to 
\begin{align*}  
\dt H\ge&\,(-F)\left(-|A|^2+\left(\tfrac{n-2}2\right)^2|X|^{-2}  
-2\alpha|X|^{-2}\langle X,\nu\rangle H\right)\umbruch\\ >&\,2\alpha  
F|X|^{-2}\langle X,\nu\rangle H
\end{align*}  
implies that $H>0$ for $0<t\le1$. 

Let $\mathcal N^c:=\left(\R^{n+1}\setminus \mathcal
N\right)\setminus\left(B_r(0)\times\R\right)$. There exists $\delta>0$
such that $|p|\cdot H[k](p)\ge\delta$ in $\mathcal N^c$.  Equation
\eqref{A evol i} implies that $\left|\dt|A|\right|\le\frac
c{|X|}\frac1{|X|}$. Hence $H[k]>0$ is preserved in $\mathcal N^c$ for
$0\le t\le1$.

 Note also that $\left|\dt X\right|$  
becomes small near infinity. Hence $|X|$ hardly changes during the  
evolution. The evolution equation for $\nu$ implies that the normal  
hardly changes during the evolution. Hence $X(\cdot,1)$ can be written  
as graph of a smooth function. Moreover $b<k$ as $\nu$ hardly changes.  
Bounds on $k-b$ are immediate from the decay of $|F|$. The Lemma  
follows.  
\end{proof}  
 
\begin{lemma}\label{t delta ex lem} 
Let $k$ be as in Lemma \ref{barrier}. Let $$u\in
C^\infty_{loc}\left(\R^n\times(0,\infty)\right)\cap
C^0_{loc}\left(\R^n\times[0,\infty)\right)$$ be a solution to
\eqref{mcf} such that $u_0:=u(\cdot,0)$ fulfills \eqref{close to k
eq}.  Let $\delta>0$.  Then there exists $t_\delta>0$ such that
$$u(\cdot,t_\delta)\ge k-\delta.$$ 
\end{lemma}   
\begin{proof} 
Fix $R>0$ such that $|u_0-k(x)|<\frac\delta2$ for $|x|\ge R$. Let 
$m:=-\inf\,(u_0-k)$. We may assume that $m>\delta$. Let $b$ be the 
barrier obtained in Lemma \ref{barrier}. Assume that $b$ is defined in 
$\R^n\setminus B_{R_1}(0)$ and set $m_1:=-\inf\,(b-k)$. Observe that for 
$\lambda>0$ 
$$b^\lambda(x)=\lambda b\left(\frac x\lambda\right)$$ also fulfills
the conditions on $b$ in Lemma \ref{barrier}. Choose $\lambda>0$ such
that $\lambda m_1>m$ and $\lambda R_1>R$. Hence
$B(x,t):=\max\left\{U(x,t)-m,b^\lambda(x)-\frac\delta2\right\}$ (and
$B(x,t):=U(x,t)-m$, where $b^\lambda$ is not defined) is a subsolution
to \eqref{mcf} in the viscosity sense. As $B(\cdot,t)$ is asymptotic
to a cone, the maximum principle \cite[Theorem
2.1]{BarlesetalUniqueness} implies that $u(\cdot,t)\ge B(\cdot,t)$ for
all $t\ge0$. There exists $r>0$ such that $b\ge k-\delta$ on
$\R^n\setminus B_r(0)$. Positivity of $U(\cdot,t)-k$ for any $t>0$ and
\eqref{U scaling} imply that there exists $t_\delta>0$ such that
$U(x,t)-m\ge k(x)$ for $|x|\le r$ and all $t\ge t_\delta$. We obtain
$u(\cdot,t)\ge k-\delta$ for all $t\ge t_\delta$.
\end{proof} 
 
\begin{proof}[Proof of Theorem \ref{main thm}, 2nd part:]  
Here we will prove Theorem \ref{main thm} under the assumption that
$k$ is not a linear function, see Assumption \eqref{ass i} in the
first part of the proof. This case is considered in Appendix \ref{Rn
stable}. We will also assume that $n\ge3$. If $n=1$ or $n=2$, $k$ is
convex. We will consider this case independently.
 
According to Corollary \ref{U strict above k}, $U(\cdot,t)-k>0$ for 
$t>0$. Lemma \ref{t delta ex lem} implies that $t_\delta$ as in 
Assumption \eqref{ass ii} in the first part of the proof of Theorem 
\ref{main thm} exists. Hence the result follows from the proof given 
there. 
\end{proof}  
 
The proof becomes simpler if the function $k$ is convex.  
\begin{proof}[Proof of Theorem \ref{main thm}, 3rd part:] 
Here we assume that $k$ is convex.  This is obvious if $n=1$. If
$n=2$, we observe, that outside the origin, one principal curvature of
a cone vanishes. Hence $H\ge0$ is equivalent to local convexity
outside the origin. As $k$ is a cone, by continuity, the function $k$
is convex on all of $\R^n$.
 
For any point in $\graph k$, there exists a supporting hyperplane,
which we assume to be $\graph l$ for some linear function $l$. This
hyperplane is a supporting hyperplane for a half-line in $\graph
k$. Hence for every $\delta>0$ the considerations in Appendix \ref{Rn
stable} imply the existence of $t_\delta>0$ such that $u(\cdot,t)\ge
l-\delta$.  As the decay assumption \eqref{close to k eq} is
independent of the direction in which we approach infinity, the
results of Appendix \ref{Rn stable} and considerations as in the proof
in Appendix \ref{uniform appendix}, applied to the results in Appendix
\ref{Rn stable} after an appropriate rotation, imply that $t_\delta>0$
can be chosen independently of $l$. Hence Assumption \eqref{ass ii} in
the first part of the proof of Theorem \ref{main thm} is fulfilled. We
can now proceed as in the first part of the proof.
\end{proof} 
 
\begin{remark}
If $k\le u_0\le U(\cdot,T)$ for some $T>0$, then the proof of Theorem
\ref{main thm} implies the decay rate
$$\sup\limits_{\R^n}|u(\cdot,t)-U(\cdot,t)|\le\tfrac c{\sqrt t}.$$
According to the considerations in the proof of Lemma \ref{U converges
lem}, this rate is sharp.
\end{remark}

\section{BV-Perturbations} 
\label{BV sec} 
 
In this section we discuss how to replace \eqref{close to k eq} by a
weaker condition that allows for additional decaying perturbations in
$BV_{loc}$ in the case $n\ge2$.
 
In the following, it is possible to consider $u\in 
C^0_{loc}\left(\R^n\right) \cap BV_{loc}\left(\R^n\right)$. For the 
sake of an easier presentation, however, we will assume that $u\in 
C^1_{loc}\left(\R^n\right)$. Set $\Vert 
u\Vert_{BV(\Omega)}:=\int\limits_\Omega|u|+|Du|$.  We generalize 
\eqref{close to k eq} as follows: Assume that there exists 
$\epsilon:\R_+\to\R_+$, depending on $u$, such that $\epsilon(r)\to0$ 
as $r\to\infty$ and for all $r>0$  
\begin{equation}\label{close to k plus BV} 
\sup\limits_{x\in\R^n\setminus B_r(0)}\Vert 
u_0-k\Vert_{BV(B_1(x)\cap\{|u_0-k|>\epsilon(r)\})}<\epsilon(r). 
\end{equation} 
 
In order to show that our results remain valid under this initial 
assumption, it suffices to prove that \eqref{close to k plus BV} 
implies a condition of the form \eqref{close to k eq} for some 
positive time. More precisely, it suffices to show that for every 
$\delta>0$ there exist $r>0$ and $t_0=t_0>0$, both depending on $u$, 
such that 
\begin{equation}\label{close to k with t} 
|u(\cdot,t_0)-k|<\delta\quad\text{in}\quad \R^n\setminus B_r(0). 
\end{equation} 
This is a consequence of Brakke's clearing out lemma
\cite{BrakkeBook}:
\begin{lemma}\label{A lem}
Let $k:\R^n\to\R$ be a cone which is smooth outside the origin.  Let 
$u_0\in C^1_{loc}\left(\R^n\right)$. Assume that there exists 
$\epsilon:\R_+\to\R_+$ such that $u_0$ fulfills \eqref{close to k plus 
BV}. Let $u\in C^\infty_{loc}\left(\R^n\times(0,\infty)\right)\cap 
C^0_{loc}\left(\R^n\times[0,\infty)\right)$ be a solution to 
\eqref{mcf} with $u(\cdot,t)=u_0$. Let $\delta>0$. Then there exist 
$t_0>0$ and $r>0$ such that \eqref{close to k with t} is fulfilled. 
\end{lemma} 
\begin{proof} 
Fix $x\in\R^n$ and $0<\rho<1$. Let us assume that $\sup|Dk|\le G$
for some $G\ge1$. Then we get $k(y)\le k(x)+\rho G$ for $y\in
B_\rho(x)$. Hence the ball $B_\rho(x,k(x)+2\rho+\rho G)
\subset\R^{n+1}$ lies above $k+\rho$. Let us estimate the area $A$ of
$\graph u_0$ above $k+\rho$ in $B_\rho(x)\times\R$.
Set $\Omega:=\{y\in B_\rho(x)\,:\,u_0(y)>k(y)+\rho\}$. We have
\begin{align*} 
A=&\,\int\limits_{\Omega} \sqrt{1+|Du_0|^2}
\le\int\limits_{\Omega\cap\{|Du_0|\ge2G\}} 2|Du_0|
+\int\limits_{\Omega\cap\{|Du_0|<2G\}} \sqrt{1+4G^2}\umbruch\\
\le&\,\int\limits_{\Omega\cap\{|Du_0|\ge2G\}} 4|D(u_0-k)|
+\int\limits_{\Omega\cap\{|Du_0|<2G\}} 3G\frac{|u_0-k|}\rho\umbruch\\
\le&\,\left(4+\tfrac{3G}\rho\right)\Vert u_0-k\Vert_{BV(\Omega)}
\le\left(4+\tfrac{3G}\rho\right)\Vert u_0-k\Vert_{BV(B_1(x)\cap
\{|u_0-k|>\epsilon(|x|)\})}
\end{align*} 
for $\epsilon=\epsilon(|x|)<\rho$. According to \eqref{close to k plus
BV}, the right-hand side is small for all $|x|$ sufficiently
large. Hence the clearing out lemma, \cite[Lemma 6.3]{BrakkeBook} or
\cite[Proposition 4.23]{KEMCFBook}, implies that $(x,k(x)+2\rho+\rho
G)\not\in\graph u(\cdot,t_0)$ for some $t_0=c\rho^2>0$. A similar
argument ensures that $\{x\}\times((-\infty,k(x)-2\rho-\rho
G)\cup(k(x)+2\rho+\rho G,\infty))\cap\graph u(\cdot,t_0)=\emptyset$
for all $|x|$ sufficiently large. The claim follows.
\end{proof} 
 
All the arguments in the proof of Theorem \ref{main thm} extend to
such perturbations $u$ if we use \eqref{close to k with t} instead of
\eqref{close to k eq}.

Observe also that instead of \eqref{close to k plus BV} we could
require directly that the area $A$ of $\graph u_0$ estimated in the
proof of Lemma \ref{A lem} is small enough to apply the clearing out
lemma.
 
\begin{appendix}  
\section{Stability of $\R^n$ under Mean Curvature Flow}  
\label{Rn stable}  
  
\begin{remark} 
In this section, we will always assume that our solutions to mean 
curvature flow are graphical. For other complete solutions with 
corresponding behavior at infinity, such theorems are also true if the 
solutions exist for $t\in[0,\infty)$. A graphical solution, which is 
initially above the considered solution, shifted by $\epsilon$, can be 
used as a barrier. This implies these more general results. 
\end{remark} 
 
In the following theorem, we show that a solution to \eqref{mcf} 
leaves every half-space if near infinity it is initially not too far 
inside the half-space. The following proof is due to G. Huisken.  
\begin{theorem}\label{half space avoid thm} 
Let $u_0:\R^n\to\R$ be continuous and assume that for every 
$\epsilon>0$ there exists $r>0$ such that $u_0<\epsilon$ in 
$\R^n\setminus B_r(0)$. Let $u\in 
C^\infty_{loc}\left(\R^n\times(0,\infty)\right) \cap 
C^0_{loc}\left(\R^n\times[0,\infty)\right)$ be a solution to 
\eqref{mcf}. Then 
$$\limsup\limits_{t\to\infty}\sup\limits_{\R^n} u(\cdot,t)\le 0.$$ 
\end{theorem} 
 
\begin{proof}
Consider $\Phi(x,t):=\frac1{(4\pi t)^{n/2}}e^{-\frac{|x|^2}{4t}}$. Set 
$M_t:=\graph u(\cdot,t)$. We claim that 
$$\dt\Phi(X,t)-\Delta^{M_t}\Phi(X,t) 
=\dt\Phi(X,t)-g^{ij}(\Phi(X,t))_{;ij}\ge0.$$ This is equivalent to 
$$\dt\Psi(X,t)-g^{ij}(\Psi(X,t))_{;ij}
-g^{ij}(\Psi(X,t))_i(\Psi(X,t))_j\ge0,$$ where
$$\Psi=\Psi(X,t)=\log\Phi(X,t)+\frac n2\log(4\pi)=-\frac n2\log 
t-\frac{|X|^2}{4t}.$$ We obtain 
\begin{align*} 
\dt\Psi(X,t)=&\,-\frac n{2t}+\frac{|X|^2}{4t^2}+\frac{\langle 
X,\nu\rangle H}{2t},\umbruch\\ \Psi_i=&\,-\frac{X^\alpha\ol 
g_{\alpha\beta}X^\beta_i}{2t},\umbruch\\ 
\Psi_{;ij}=&\,-\frac{g_{ij}}{2t}+\frac{\langle X,\nu\rangle 
h_{ij}}{2t},\umbruch\\ 
\dt\Psi-g^{ij}\Psi_{;ij}-g^{ij}\Psi_i\Psi_j=&\,\frac{|X|^2}{4t^2} 
-\frac{X^\alpha\ol g_{\alpha\beta}X^\beta_ig^{ij}X^\gamma_j\ol 
g_{\gamma\delta} X^\delta}{4t^2}\umbruch\\ 
=&\,\frac1{4t^2}\left(|X|^2-X^\alpha\ol g_{\alpha\beta}\left(\ol 
g^{\beta\gamma}-\nu^\beta\nu^\gamma\right)\ol g_{\gamma\delta} 
X^\delta\right)\umbruch\\ 
=&\,\frac{\langle X,\nu\rangle^2}{4t^2}\ge0. 
\end{align*} 
Define $\eta:=(0,\ldots,0,1)$. Then $u^X:=\eta_\alpha X^\alpha$ equals
$u$ up to a change in the parametrization. We get
$$\dt u^X-g^{ij}(u^X)_{;ij}
=\eta_\alpha(-H\nu^\alpha+g^{ij}h_{ij}\nu^\alpha)=0.$$ Using large
balls as barriers, we see that for every $T>0$ and every $\epsilon>0$,
there exists $r>0$ such that $u^X\le\epsilon/2$ on $\graph
u(\cdot,t)|_{\R^n\setminus B_r(0)}$, $t\in[0,T]$. Fix $t_0>$ small and
$\epsilon>0$. As $\Phi(\cdot,t)$ is a continuous positive function,
there exists $a>0$ such that $a\Phi(X,t_0) +\epsilon
-u^X(\cdot,t_0)\ge0$ in $B_r(0)\times\R$. The considerations above
ensure that $a\Phi(X,t)+\epsilon-u^X(X,t)\ge\epsilon/2$ for $t$ in a
bounded time interval and $X\in\left(\R^n\setminus
B_r(0)\right)\times\R$ for $r$ sufficiently large, depending in
particular on the time interval considered. As
$$\dt\left(a\Phi-u^X\right)-g^{ij}\left(a\Phi-u^X\right)_{;ij}\ge0,$$
the maximum principle implies that $a\Phi+\epsilon-u^X\ge0$ for all
$t\ge t_0$.  As $a\Phi\to0$, uniformly as $t\to\infty$, we deduce that
$$\limsup\limits_{t\to\infty}\sup_{x\in\R^n}u(x,t)\le2\epsilon.$$ 
The claim follows.  
\end{proof} 
 
As a corollary to Theorem \ref{half space avoid thm} we obtain the 
following stability theorem, which generalizes \cite[Appendix 
C]{JCOSFSMCFStability} to all dimensions. 
\begin{theorem}\label{Rn stable thm} 
Let $u_0\in C^0_{loc}\left(\R^n\right)$ fulfill \eqref{close to k 
eq}. Let $u\in C^\infty_{loc}\left(\R^n\times(0,\infty)\right) \cap 
C^0_{loc}\left(\R^n\times[0,\infty)\right)$ be a solution to 
\eqref{mcf}. Then 
$$\sup\limits_{\R^n}|u(\cdot,t)|\to0\quad\text{as} \quad t\to\infty.$$ 
Moreover, $$\Vert D^\alpha 
u(\cdot,t)\Vert_{L^\infty\left(\R^n\right)}\to0 \quad\text{as}\quad 
t\to\infty$$ for every derivative $D^\alpha u$. 
\end{theorem} 
\begin{proof} 
Theorem \ref{half space avoid thm} yields the upper bound; the lower 
bound follows similarly. This implies $C^0$-convergence. Uniform 
interior derivative estimates and interpolation inequalities yield 
$C^k$-convergence for any $k\in\N$. 
\end{proof}  
 
\section{Uniform Stability for Families of Solutions}  
\label{uniform appendix}  
 
Our stability result extends to families of solutions as follows. 
\begin{theorem}  
Let $k$ and $U$ be as in Theorem \ref{main thm}.  Let $(u^i)_i$ be a  
family of functions such that each function fulfills the conditions on  
$u$ in Theorem \ref{main thm}. If \eqref{close to k eq} is uniformly  
fulfilled in the sense that  
$$\sup\limits_i\sup\limits_{\R^n\setminus
B_r(0)}|u^i(\cdot,0)-k|\to0\quad \text{as}\quad r\to\infty,$$ then
$$\sup\limits_i \left\Vert D^\alpha(u^i-U)(\cdot,t)
\right\Vert_{L^\infty\left(\R^n\right)}\to0 \quad\text{as}\quad
t\to\infty,$$ for any derivative $D^\alpha$.
\end{theorem}  
\begin{proof}  
Fix $\epsilon>0$.  Use $u^++\epsilon$ and $u^--\epsilon$ as barriers  
for $u^i$, where $u^\pm$ are solutions to \eqref{mcf} such that  
$u^+(\cdot,0)\ge u^i(\cdot,0)\ge u^-(\cdot,0)$ and $u^\pm(\cdot,0)$  
fulfill the assumptions on $u_0$ of Theorem \ref{main thm}.  
\end{proof}  
  
\end{appendix}  
  
\bibliographystyle{amsplain}  

\begin{thebibliography}{10}

\bibitem{BarlesetalUniqueness}
Guy Barles, Samuel Biton, Mariane Bourgoing, and Olivier Ley, \emph{Uniqueness
  results for quasilinear parabolic equations through viscosity solutions'
  methods}, Calc. Var. Partial Differential Equations \textbf{18} (2003),
  no.~2, 159--179.

\bibitem{OSPierreCrelle}
Pierre Bayard and Oliver~C. Schn\"urer\weg{OSPierreCrelle}, \emph{{Entire
  spacelike hypersurfaces of constant Gau\ss{} curvature in Minkowski space}},
  J. Reine Angew. Math., to appear.

\bibitem{BrakkeBook}
Kenneth~A. Brakke, \emph{The motion of a surface by its mean curvature},
  Mathematical Notes, vol.~20, Princeton University Press, Princeton, N.J.,
  1978.

\bibitem{CaffarelliHardtSimonSing}
Luis Caffarelli, Robert Hardt, and Leon Simon, \emph{Minimal surfaces with
  isolated singularities}, Manuscripta Math. \textbf{48} (1984), no.~1-3,
  1--18.

\bibitem{OSAlbert}
Albert Chau and Oliver~C. Schn\"urer, \emph{{S}tability of gradient
  {K}\"ahler-{R}icci solitons}, Comm. Anal. Geom. \textbf{13} (2005), no.~4,
  769--800.

\bibitem{JCOSFSMCFStability}
Julie Clutterbuck, Oliver~C. Schn\"urer, and Felix
  Schulze\weg{JCOSFSMCFStability}, \emph{Stability of translating solutions to
  mean curvature flow}, Calc. Var. Partial Differential Equations \textbf{29}
  (2007), no.~3, 281--293.

\bibitem{JCPhD}
Julie Clutterbuck\weg{JCPhD}, \emph{{Parabolic equations with continuous
  initial data}}, Ph.D. thesis, Australian National University, 2004, {\tt
  arXiv:math.AP/0504455}.

\bibitem{ColdingMinicozzi}
Tobias~H. Colding and William~P.\weg{ColdingMinicozzi} Minicozzi, II,
  \emph{Sharp estimates for mean curvature flow of graphs}, J. Reine Angew.
  Math. \textbf{574} (2004), 187--195.

\bibitem{KEMCFBook}
Klaus Ecker, \emph{Regularity theory for mean curvature flow}, Progress in
  Nonlinear Differential Equations and their Applications, 57, Birkh\"auser
  Boston Inc., Boston, MA, 2004.

\bibitem{EckerHuiskenInvent}
Klaus Ecker and Gerhard\weg{EckerHuiskenInvent} Huisken, \emph{Interior
  estimates for hypersurfaces moving by mean curvature}, Invent. Math.
  \textbf{105} (1991), no.~3, 547--569.

\bibitem{EvansPDE}
Lawrence~C. Evans, \emph{Partial differential equations}, Graduate Studies in
  Mathematics, vol.~19, American Mathematical Society, Providence, RI, 1998.

\bibitem{CGCPBook}
Claus\weg{CGCPBook} Gerhardt, \emph{Curvature problems}, Series in Geometry and
  Topology, vol.~39, International Press, Somerville, MA, 2006.

\bibitem{HuiskenRoundSphere}
Gerhard Huisken, \emph{Flow by mean curvature of convex surfaces into spheres},
  J. Differential Geom. \textbf{20} (1984), no.~1, 237--266.

\bibitem{HuiskenSinestrari3}
Gerhard Huisken and Carlo\weg{HuiskenSinestrari3} Sinestrari, \emph{Mean
  curvature flow with surgeries of two-convex hypersurfaces}, 2006, Invent.
  Math., to appear.

\bibitem{OSFSMSStabilityRicci}
Oliver~C. Schn\"urer, Felix Schulze, and Miles Simon\weg{OSFSMSStabilityRicci},
  \emph{{Stability of Euclidean space under Ricci flow}}, Comm. Anal. Geom.
  \textbf{16} (2008), no.~1, 127--158.

\bibitem{OSA2}
Oliver~C. Schn\"urer\weg{OSA2}, \emph{Surfaces contracting with speed
  $|{A}|^2$}, J. Differential Geom. \textbf{71} (2005), no.~3, 347--363.

\bibitem{StavrouSelfSim}
Nikolaos\weg{StavrouSelfSim} Stavrou, \emph{Selfsimilar solutions to the mean
  curvature flow}, J. Reine Angew. Math. \textbf{499} (1998), 189--198.

\bibitem{WhiteSize}
Brian White, \emph{The size of the singular set in mean curvature flow of
  mean-convex sets}, J. Amer. Math. Soc. \textbf{13} (2000), no.~3, 665--695.

\end{thebibliography}

\def\weg#1{} \def\unterstrich{\underline{\rule{1ex}{0ex}}} \def\cprime{$'$}
  \def\cprime{$'$} \def\cprime{$'$} \def\cprime{$'$}
\providecommand{\bysame}{\leavevmode\hbox to3em{\hrulefill}\thinspace}
\providecommand{\MR}{\relax\ifhmode\unskip\space\fi MR }
\providecommand{\MRhref}[2]{%
  \href{http://www.ams.org/mathscinet-getitem?mr=#1}{#2}
}
\providecommand{\href}[2]{#2}

\end{document}